\documentclass[12pt,en]{elegantpaper}
\usepackage{extarrows}
\usepackage{mathrsfs}
\numberwithin{equation}{section}
\allowdisplaybreaks[4]
\everymath{\displaystyle}

\title{The Liouville-type equation and an Onofri-type inequality on closed 4-manifolds\footnotemark[2]}

\author{Xi-Nan Ma \and Tian Wu\footnotemark[1] \and Xiao Zhou}

\begin{document}

\date{}
\maketitle

\renewcommand{\thefootnote}{\fnsymbol{footnote}}

\footnotetext[1]{The corresponding author.}

\begin{abstract}

    In this paper, we study the Liouville-type equation
    \[\Delta ^2 u-\lambda_1\kappa\Delta u+\lambda_2\kappa^2(1-\mathrm e^{4u})=0\]
    on a closed Riemannian manifold \((M^4,g)\) with \(\operatorname{Ric}\geqslant 3\kappa g\) and \(\kappa>0\). Using the method of invariant tensors, we derive a differential identity to classify solutions within certain ranges of the parameters \(\lambda_1,\lambda_2\). A key step in our proof is a second-order derivative estimate, which is established via the continuity method. As an application of the classification results, we derive an Onofri-type inequality on the 4-sphere and prove its rigidity.
\end{abstract}

\tableofcontents
\section{Introduction}\label{sec:intro}

Let \((M^n,g)\) be a smooth, connected, and closed (i.e., compact and without boundary) Riemannian manifold. The Laplace-Beltrami operator is defined as \(\Delta:=\mathrm{div} \nabla\). Denote the Ricci curvature as \(\mathrm{Ric}\), and assume \(\mathrm{Ric} \geqslant (n-1)\kappa g\) with a positive constant \(\kappa\). We always equip \(M\) with its normalized measure \({\mathrm{dvol}}_g\), such that \(\mathrm{vol}(M)=\int_M{\mathrm{dvol}}_g=1\).

We use \((\mathbb{S}^n(r),g_c)\) to denote the Euclidean sphere of radius \(r\) in \(\mathbb R^{n+1}\), equipped with the canonical metric \(g_c\) induced from the embedding \(\mathbb{S}^n(r)\hookrightarrow \mathbb{R}^{n+1}\). The unit sphere is \(\mathbb{S}^n:=\mathbb{S}^n(1)\). Analogously, denote the Euclidean ball of radius \(r\) in \(\mathbb{R}^n\) by \(\mathbb{B}^{n}(r)\), and the unit ball is \(\mathbb{B}^{n}:=\mathbb{B}^{n}(1)\). Denote \(\omega_n=\mathcal{H}^{n-1}(\mathbb{S}^{n-1})=\frac{2\pi ^{\frac{n}{2}}}{\Gamma(\frac{n}{2})}\).

We consider the fourth-order Liouville-type equation
\begin{equation}\label{eq:Liouville}
    \Delta ^2 u - \lambda_1 \kappa \Delta u + \lambda_2 \kappa^2 (1-\mathrm e^{4u})=0\quad\text{on }(M^4,g),
\end{equation}
where \(\lambda_2>0\). We obtain a rigidity result as follows.

\begin{definition}\label{def:Liouville}
    Define two functions
    \begin{align*}
        A_1(x)=&~252756288 - 470882880x + 170417232x^2 + 11016224x^3 \\
        & -18178788x^4 + 3558300x^5-224785x^6,
    \end{align*}
    \begin{align*}
        A_2(x)=&~619127091200 - 710561935360x + 194617078784x^2 + 70308920320x^3 \\
        &- 43978140992x^4 + 3834335296x^5 + 1625924432x^6 \\
        &- 352127584x^7 + 4233764x^8 + 4104276x^9 - 281961x^{10}.
    \end{align*}

    For \(\frac{2}{129}(\sqrt{18673}-64)\leqslant x \leqslant 2\), define \(L_1(x)=4+x\).

    For \(-2\leqslant x < \frac{2}{129}(\sqrt{18673}-64)\), define
    \begin{align*}
        L_1(x)=\frac{2+x}{56(14-x)}\Big[996+332x-67x^2&+\sqrt[3]{A_1(x)+42(14-x)\sqrt{3A_2(x)}}\\
        &+\sqrt[3]{A_1(x)-42(14-x)\sqrt{3A_2(x)}}\Big].
    \end{align*}
\end{definition}
\begin{remark}
    \(L_1 \in C[-4,2],~L_1(-2)=0,~L_1(2)=6\).
\end{remark}

\begin{theorem}\label{thm:Liouville}
    Suppose that \((M^4,g)\) is a smooth, closed Riemannian manifold with \(\mathrm{Ric}\geqslant 3\kappa g\). For \(-2 < \lambda_1 \leqslant 2\) and \(0<\lambda_2\leqslant L_1(\lambda_1)\), any \(C^4\) solution to \eqref{eq:Liouville} must be 0, unless \(\lambda_1=2\), \(\lambda_2=6\), \((M^4,g)\) is isometric to \(\Big({\mathbb{S}}^4(\frac{1}{\sqrt{\kappa}}),g_c\Big)\), and there exists \(a\in {\mathbb{S}}^4(\sqrt{\kappa}),~t>0\) such that
    \begin{equation}\label{sol:Liouville}
        u(x)=-\log(\cosh t+\langle a,x\rangle_{\mathbb{R}^5} \sinh t),~x\in {\mathbb{S}}^4(\frac{1}{\sqrt{\kappa}}).
    \end{equation}
\end{theorem}

As an application, we obtain an Onofri-type inequality on the 4-sphere and its rigidity. This result extend Beckner's result in \cite{Bec93}, which corresponds to the case \(\lambda=\frac{1}{48}\) on 4-spheres.
\begin{theorem}\label{thm:Onofri}
    For any \(f\in H^2\Big({\mathbb{S}}^4(\frac{1}{\sqrt{\kappa}})\Big)\), \(\lambda\geqslant\frac{1}{48}\), the Onofri-type inequality
    \begin{equation}\label{ineq:Onofri}
        \log \int_{{\mathbb{S}}^4(\frac{1}{\sqrt{\kappa}})} \mathrm{e}^f - \int_{{\mathbb{S}}^4(\frac{1}{\sqrt{\kappa}})} f \leqslant \frac{\lambda}{\kappa^2}\int_{{\mathbb{S}}^4(\frac{1}{\sqrt{\kappa}})} |\Delta f|^2 + \frac{1}{\kappa} (\frac{1}{8}-4\lambda) \int_{{\mathbb{S}}^4(\frac{1}{\sqrt{\kappa}})} |\nabla f|^2.
    \end{equation}
    holds. Moreover, unless \(f\) is constant, the equality is attained iff \(\lambda=\frac{1}{48}\), with
    \[f(x)=-4\log(\langle a,x\rangle_{\mathbb{R}^5} +c),~x\in {\mathbb{S}}^4(\frac{1}{\sqrt{\kappa}}),~a\in \mathbb{R}^5 \setminus \{0\},~c>\frac{|a|}{\sqrt{\kappa}}.\]
\end{theorem}
\begin{remark}
    Theorem \ref{thm:Onofri} corresponds to the equation \eqref{eq:Liouville} by taking \(\lambda=\frac{1}{8(4+\lambda_1)}\), which is to say, the Onofri-type inequality \eqref{ineq:Onofri} is equivalent to 
    \begin{equation}\label{ineq:Onofri-lambda}
        8\lambda_2 \kappa^2 \Big(\log \int_{{\mathbb{S}}^4(\frac{1}{\sqrt{\kappa}})} \mathrm{e}^f - \int_{{\mathbb{S}}^4(\frac{1}{\sqrt{\kappa}})} f \Big) \leqslant \int_{{\mathbb{S}}^4(\frac{1}{\sqrt{\kappa}})} |\Delta f|^2 + \lambda_1 \kappa \int_{{\mathbb{S}}^4(\frac{1}{\sqrt{\kappa}})} |\nabla f|^2,
    \end{equation}        
    where \(-4 < \lambda_1 \leqslant 2,~0<\lambda_2\leqslant 4+\lambda_1\). We can observe from the proof of Theorem \ref{thm:Onofri} that for a closed manifold \((M^4,g)\), if the Onofri-type inequality is true when \(\lambda=\frac{1}{48}\), then the rigidity result also holds on this manifold. Besides, the Onofri-type inequality \eqref{ineq:Onofri} degenerates to the Poincaré inequality \eqref{ineq:Poincaré-H2} as \(\lambda\to\infty\).
\end{remark}

Figure \ref{fig:Liouville} describes ranges of \((\lambda_1,\lambda_2)\) associated with two theorems. The rigidity result of the equation \eqref{eq:Liouville} holds in the green region, defined by \(-2<\lambda_1\leqslant 2\) and \(0<\lambda_2\leqslant L_1 (\lambda_1)\). The Onofri inequality \eqref{ineq:Onofri-lambda} holds in both the red and the green region, defined by \(-4<\lambda_1\leqslant 2\) and \(0<\lambda_2\leqslant 4+\lambda_1\). It's nature to raise the following conjecture about the red region.

\begin{conjecture}
    If \(-4<\lambda_1\leqslant-2\) with \(0 <\lambda_2\leqslant 4+\lambda_1\), or \(-2\leqslant\lambda_1<\frac{2}{129}(\sqrt{18673}-64)\) with \(L_1(\lambda_1)<\lambda_2\leqslant 4+\lambda_1\), then any \(C^4\) solution to \eqref{eq:Liouville} must be 0.
\end{conjecture}

\begin{figure}
    \centering
    \includegraphics[width=0.75\linewidth]{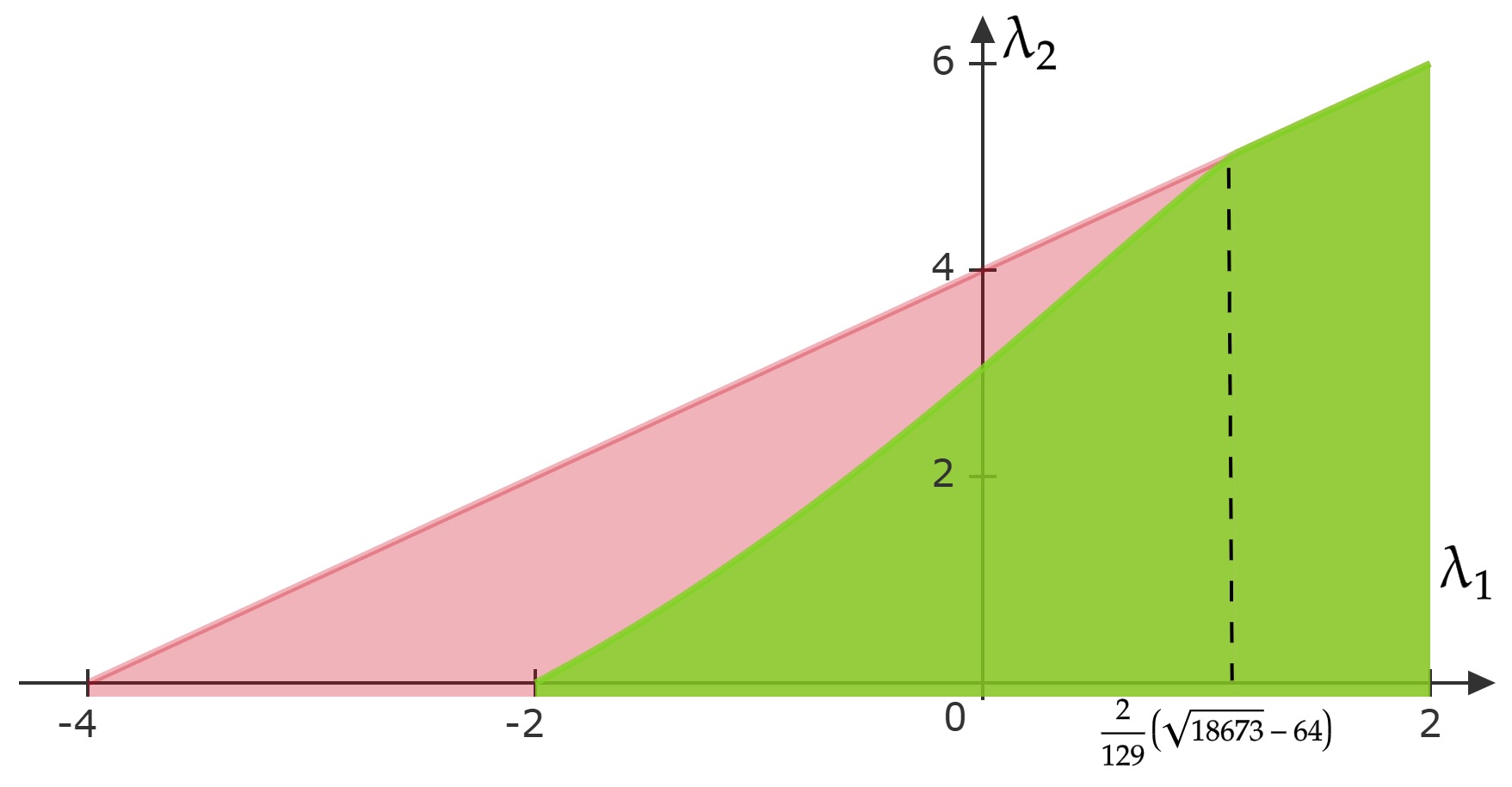}
    \caption{The classification result of the equation \eqref{eq:Liouville} holds in the green region defined by \(-2<\lambda_1\leqslant 2,~0<\lambda_2\leqslant L_1 (\lambda_1)\). The Onofri inequality \eqref{ineq:Onofri-lambda} holds in the red and the green region defined by \(-4<\lambda_1\leqslant 2,~0<\lambda_2\leqslant 4+\lambda_1\)}
    \label{fig:Liouville}
\end{figure}

When \(\lambda_1=2\) and \(\lambda_2=6\), the equation \eqref{eq:Liouville} corresponds to the Q-Yamabe problem, that is, finding conformal metrics with constant Q-curvature. Precisely, it is deduced from the transformation law for the Q-curvature under a conformal change of metric 
\begin{equation*}
\begin{cases}
\mathrm{P}_g u+\mathrm{Q}_g=\mathrm{Q}_{\mathrm{e}^{2u}g}\mathrm{e}^{4u}\quad\text{on }(M^4,g)\\
\mathrm{P}_g u=\frac{n-4}{2}\mathrm{Q}_{u^{\frac{4}{n-4}}g} u^{\frac{n+4}{n-4}},~u>0\quad\text{on }(M^n,g)\quad\text{if }n\neq 4
\end{cases}
\end{equation*}
where \(\mathrm{P}_g\) is the Paneitz-Branson operator (see \cite{Bra85}) defined as
\[\mathrm{P}_g :=\Delta_g^2-\operatorname{div}_g\Big(\Big(\frac{n^2-4n+8}{2(n-1)(n-2)}\mathrm{S}_g g-\frac{4}{n-2}\mathrm{Ric}_g\Big)\nabla u \Big)+\frac{n-4}{2}\mathrm{Q}_g,
\]
and \(\mathrm{Q}_g\) is the Q-curvature defined as
\[\mathrm{Q}_g :=\frac{1}{2(n-1)}\Delta_g \mathrm{S}_g +\frac{n^3-4n^2+16n-16}{8(n-1)^2(n-2)^2}{\mathrm{S}_g}^2 -\frac{2}{(n-2)^2}|\mathrm{Ric}_g|^2.
\]
As a higher dimension version \(n\geqslant5\), the corresponding equation is
\begin{equation}\label{eq:subcritical}
    \Delta ^2 u-\lambda_1\kappa\Delta u+\lambda_2\kappa^2(u-u^\alpha)=0\quad\text{on }(M^n,g),
\end{equation}
where \(1<\alpha\leqslant\frac{n+4}{n-4}\). When \(\lambda_1=\frac{n^2-2n-4}{2}\) and \(\lambda_2=\frac{n(n-2)(n+2)(n-4)}{16}\), the equation \eqref{eq:subcritical} corresponds to the Q-Yamabe problem. Rigidity results of the equation \eqref{eq:subcritical} are studied by the authors in \cite{MWZ25}.

Let's recall related second-order semilinear equations in \(\mathbb R^n\). Gidas-Spruck \cite{GS81} proved non-existence of positive solution to \(-\Delta u=u^{\alpha}\) in subcritical case \(1<\alpha<\frac{n+2}{n-2}\) with \(n\geqslant 3\). In the critical case \(\alpha=\frac{n+2}{n-2}\), Caffarelli-Gidas-Spruck \cite{CGS89} (also see Chen-Li \cite{CL91}) classified positive solutions via the moving plane method. Chen-Li \cite{CL91} classified solutions to the Liouville equation \(-\Delta u=\mathrm e^{2u}\) in \(\mathbb R^2\) under the finite-energy condition.

On closed manifolds with \(\operatorname{Ric}\geqslant(n-1)g\), corresponding second-order equations were also studied widely via the differential identities method. Obata \cite{Oba71} classified all positive solutions to \(\Delta u-\frac{n(n-2)}{4}(u-u^{\alpha})=0\), and gave the uniqueness result of the Yamabe problem. Bidaut-Véron and Véron \cite{VV91} extended Gidas-Spruck's results by improving Obata's identity and proved that positive solutions of \(\Delta u-\lambda(u-u^{\alpha})=0\) must be 1 in the subcritical case \(1<\alpha<\frac{n+2}{n-2}\) with \(0<\lambda\leqslant\frac{n}{\alpha-1}\). This result improved the Sobolev inequality as
\[\Big(\int_M|f|^q\Big)^{\frac 2 q}-\int_M f^2\leqslant\frac{q-2}{n}\int_M|\nabla f|^2,\quad f\in H^1(M),\quad 2\leqslant q\leqslant 2^*=\frac{2n}{n-2}.\]
By substituting \(f\) with \(1+\frac g q\) and letting \(q\to\infty\) on 2-manifolds, the sharp Onofri inequality
\[\log\int_M\mathrm{e}^f-\int_M f\leqslant\frac 1 4\int_M|\nabla f|^2,\quad f\in H^1(M),\quad n=2\]
holds. Dolbeault-Esteban-Jankowiak \cite{DEJ16} obtained the rigidity result of the Liouville equation \(\Delta u-\lambda(1-\mathrm e^{2u})=0\) on 2-manifolds with \(\lambda\leqslant1\). This result verified the rigidity of the sharp Onofri inequality. We recommend readers to see \cite{Ono82,DEJ15} for more on the Onofri inequality.

Let's focus on fourth-order equations. Lin \cite{Lin98} classified solutions to \(\Delta^2u=\mathrm e^{4u}\) in \(\mathbb R^4\) and positive solutions to \(\Delta^2u=u^\alpha\) in \(\mathbb R^n\) with \(n\geqslant5\) and \(1<\alpha\leqslant\frac{n+4}{n-4}\) by the moving plane method. He first classified the conformal metrics with constant Q-curvature in the case of the standard sphere. On closed manifolds, the following results were obtained by Vétois \cite{Vét24} and Li-Wei \cite{LW25}.

\begin{theorem}[\cite{Vét24,LW25}]\label{thm:Vetois}
    Suppose that \((M^4,g)\) is a smooth, closed, Einstein Riemannian manifold with \(\mathrm{Ric}\geqslant 3g\). If \(\kappa=1\), \(\lambda_1=2\), \(\lambda_2\leqslant 6\), and \(M\) is not conformally diffeomorphic to \(\mathbb S^4\) when \(\lambda_2=6\). Then any \(C^4\) solution to \eqref{eq:Liouville} must be 0.
\end{theorem}

\begin{theorem}[\cite{Vét24,LW25}]
    Suppose that \((M^n,g)\) is a smooth, closed, Einstein Riemannian manifold with \(\mathrm{Ric}\geqslant(n-1)g\). If \(\kappa=1\), \(1<\alpha\leqslant\frac{n+4}{n-4}\), \(\lambda_1=\frac{n^2-2n-4}{2}\), \(\lambda_2\leqslant\frac{n(n-2)(n+2)(n-4)}{16}\), and \(M\) is not conformally diffeomorphic to \(\mathbb S^4\) when \(\alpha=\frac{n+4}{n-4}\) and \(\lambda_2=\frac{n(n-2)(n+2)(n-4)}{16}\). Then any positive \(C^4\) solution to \eqref{eq:subcritical} must be 1.
\end{theorem}

These two classification results showed that for any smooth, closed Einstein manifold with positive scalar curvature and not conformally diffeomorphic to the standard sphere, the only conformal metrics to \(g\) with constant Q-curvature must be multiples of \(g\). Besides, Case \cite{Cas24} considered the mixed curvature \(\mathrm{Q}+a\sigma_2\) instead of Q-curvature, and gave a related classification result.

Hence, in Theorem \ref{thm:Liouville}, we remove the assumption of Einstein in Theorem \ref{thm:Vetois}, obtain the rigidity of manifolds when \(\lambda_1=2\) and \(\lambda_2=6\), and extend the range of \(\lambda_1\) and \(\lambda_2\). Unlike backgrounds from conformal geometry in \cite{Vét24,LW25}, we establish differential identities using the invariant tensor technique. This technique was introduced by Ma-Wu \cite{MW24} to study second-order semi-linear equations on manifolds. After which, Ma-Ou-Wu \cite{MOW25} reconstructed Jerison-Lee's identity via the invariant tensor technique and obtained a Liouville theorem for the subcritical equation on CR manifolds. Recently, Ma-Wu-Wu \cite{MWW25} extended Lin's result to non-compact Riemannian manifolds with the help of the invariant tensor technique.


In local coordinates, we denote the metric tensor as \(g_{ij}\), and the Ricci curvature tensor as \(R_{ij}\). Denote the inner product w.r.t. \(g\) as \(\langle\cdot,\cdot\rangle\). We employ the Einstein summation convention to raise and lower the indices using the metric tensor \(g_{ij}\) and its inverse \(g^{ij}\). All covariant derivatives are denoted with a comma, such as \(_{,i}\), except for acting on our solution \(u\).

In Section \ref{sec:estimate}, we establish the prior estimate of second derivatives via the maximum principle and the continuity method, where we improve the estimate in \cite{Vét24}. In Section \ref{sec:Liouville}, the differential identity is obtained by the invariant tensor technique, and its positivity is ensured by a subtle choice of parameters. Thus, Theorem \ref{thm:Liouville} is verified by the identity. In Section \ref{sec:Onofri}, as an application of Theorem \ref{thm:Liouville}, we obtain a rigidity result, Theorem \ref{thm:Onofri}, of an Onofri-type inequality.
\section{An estimate of the second order derivative}\label{sec:estimate}

In this section, \(\lambda_1\in\mathbb R\), \(f\in C^2(\mathbb R)\). We study the differential inequality
\begin{equation}\label{ineq}
    {\Delta ^2}u-\lambda_1\kappa\Delta u+f(u)\kappa^2\geqslant0.
\end{equation}
We first state the strong maximum principle that will be needed.
\begin{lemma}\label{lem:strong-extreme}
    Let \((M^n,g)\) be a Riemannian manifold, \(V\in \Gamma(TM)\), \(q\in C^2(M)\), \(q\leqslant 0\). Define the second-order elliptic linear operator of the form
    \[\mathscr Lv:=\Delta v+\langle V,\nabla v\rangle+qv,\quad \forall v \in C^2(M).\]
    For \(v\in C^2 (\operatorname{Int}M)\cap C^{0}(M)\), if \(\mathscr Lv\geqslant 0\) in \(\operatorname{Int}M\), and \(v\) attains a nonnegative maximum value in \(\operatorname{Int}M\), then \(v\) must be constant on \(M\). 
\end{lemma}

The following lemma introduces the continuity method.
\begin{lemma}\label{lem:continuity}
    Let \((M^n,g)\) be a closed Riemannian manifold, \(v_t,K_t\in C^2(M)\), \(V_t\in\Gamma(TM)\), \(K_t\geqslant0\) and \(v_t\) is continuous w.r.t. \(t\in[0,1]\). The operator \(\mathscr L_t\) is defined as
    \[\mathscr L_tv=\Delta v+\langle V_t,\nabla v\rangle+K_tv^2,\quad \forall v \in C^2(M).\]
    If \(v_0<0\), \(\mathscr L_tv_t\geqslant 0\) and \(v_t\not\equiv0\) for all \(t\in[0,1]\), then \(v_1<0\).
\end{lemma}

\begin{proof}
    Assume \(\sup\limits_M v_1\geqslant0\). By the continuity of \(v_t\) w.r.t. \(t\), \(\exists t_0 \in (0,1]\) s.t. \(\sup\limits_M v_{t_0}=0\). By Lemma \ref{lem:strong-extreme} and the compactness of \(M\), it implies \(v_{t_0}\equiv 0\), a contradiction.
\end{proof}

\begin{lemma}\label{lem:est}
    Let \(h_1,h_2,h_3\in C^2(\mathbb R;\mathbb R_+)\), \(\Phi:=h_3(u)[\Delta u+h_1(u)|\nabla u|^2-h_2(u)]\), where \(u\) is a \(C^4\) solution to \eqref{ineq}, then
    \[\frac{1}{h_3}\mathscr L\Phi
    \geqslant~h_1\Big(\frac{\Phi}{h_3},h_1|\nabla u|^2,h_2\Big)A\Big(\frac{\Phi}{h_3},h_1|\nabla u|^2,h_2\Big)^T,\]
    where \(\mathscr L\Phi:=\Delta\Phi-2\Big(h_1+\frac{h_3'}{h_3}\Big)\langle\nabla u,\nabla\Phi\rangle+K\frac{h_1}{h_3}\Phi^2\), and \(A=(A_{ij})\in\mathbb R^{3\times3}\) satisfies
    \[A_{11}=K+\frac 2 n+\frac{h_3'}{h_1h_3},\quad A_{12}=-\frac 4 n-\frac{n+4}{2n}\Big(\frac{1}{h_1}\Big)'-\frac{3h_3'}{2h_1h_3}-\frac{h_3}{2h_1^2}\Big(\frac{1}{h_3}\Big)'',\]
    \[A_{13}=\frac 2 n+\frac{\lambda_1\kappa-h_2'}{2h_1h_2}+\frac{h_3'}{2h_1h_3},\quad A_{22}=-\frac{2(n-4)}{n}-\frac{n-8}{n}(\frac{1}{h_1})'-\frac{n}{n-2}h_1^{-\frac{n+2}{n}}(h_1^{-\frac{n-2}{n}})'',\]
    \[A_{23}=-\frac 4 n-\frac{n+4}{2n}(\frac{1}{h_1})'+\frac{(2n-2-\lambda_1)\kappa+3h_2'}{2h_1h_2}-\frac{h_2''}{2h_1^2h_2},\quad A_{33}=\frac 2 n+\frac{\lambda_1\kappa-h_2'}{h_1h_2}-\frac{f(u)\kappa^2}{h_1h_2^2}.\]
    Here we omit the dependence of \(h_1,h_2,h_3\) on \((u)\) for convenience.
\end{lemma}

\begin{proof}
    By \(\Delta\nabla u=\nabla\Delta u+\mathrm{Ric}(\nabla u,\cdot)\), we obtain
    \begin{align*}
        \frac{\Delta\Phi}{h_3}=&~\Delta^2 u+2h_1|\nabla^2 u|^2+2\Big(h_1+\frac{h_3'}{h_3}\Big)\langle\nabla\Delta u,\nabla u\rangle+2h_1\mathrm{Ric}(\nabla u,\nabla u)\\
        &+4\Big(h_1'+h_1\frac{h_3'}{h_3}\Big)\langle{\nabla}^2 u,\nabla u\otimes\nabla u\rangle+h_1'|\nabla u|^2\Delta u+\Big(h_1''+2h_1'\frac{h_3'}{h_3}\Big)|\nabla u|^4\\
        &-h_2'\Delta u-\Big(h_2''+2h_2'\frac{h_3'}{h_3}\Big)|\nabla u|^2+(h_3'\Delta u+h_3''|\nabla u|^2)\frac{\Phi}{h_3^2}.
    \end{align*}
    Combining with
    \[\frac{\langle\nabla u,\nabla\Phi\rangle}{h_3}=\langle\nabla\Delta u,\nabla u\rangle+2h_1\langle{\nabla}^2 u,\nabla u\otimes\nabla u\rangle+h_1'|\nabla u|^4-h_2'|\nabla u|^2+\frac{h_3'}{h_3^2}|\nabla u|^2\Phi,\]
    the differential inequality \eqref{ineq} and the Ricci curvature condition \(\operatorname{Ric}\geqslant(n-1)\kappa g\), we have
    \begin{align*}
        &\frac{1}{h_3}\Big[\Delta\Phi-2\Big(h_1+\frac{h_3'}{h_3}\Big)\langle\nabla u,\nabla\Phi\rangle\Big]\\
        =~&2h_1[|\nabla^2 u|^2+\mathrm{Ric}(\nabla u,\nabla u)]+4(h_1'-h_1^2)\langle{\nabla}^2 u,\nabla u\otimes\nabla u\rangle+h_1'|\nabla u|^2\Delta u\\
        &+(h_1''-2h_1h_1')|\nabla u|^4-h_2'\Delta u+(2h_1h_2'-h_2'')|\nabla u|^2\\
        &+\Big\{\frac{h_3'}{h_3^2}\Delta u+\Big[\Big(\frac{h_3'}{h_3^2}\Big)'-\frac{2h_1h_3'}{h_3^2}\Big]|\nabla u|^2\Big\}\Phi+\Delta^2 u\\
        =~&2h_1\Big|\nabla^2 u+\Big(\frac{h_1'}{h_1}-h_1\Big)\nabla u\otimes\nabla u-\frac 1 n\Big[\Delta u+\Big(\frac{h_1'}{h_1}-h_1\Big)|\nabla u|^2\Big]g\Big|^2\\
        &+2h_1\mathrm{Ric}(\nabla u,\nabla u)+\frac 2 n h_1\Big[\frac{\Phi}{h_3}+\Big(\frac{h_1'}{h_1}-2h_1\Big)|\nabla u|^2+h_2\Big]^2+\frac{h_1'}{h_3}|\nabla u|^2\Phi\\
        &+\Big(h_1''-\frac{2(h_1')^2}{h_1}+h_1h_1'-2h_1^3\Big)|\nabla u|^4-\frac{h_2'}{h_3}\Phi+(h_1'h_2+3h_1h_2'-h_2'')|\nabla u|^2\\
        &-h_2h_2'+\Big\{\frac{h_3'}{h_3^3}\Phi+\Big[\Big(\frac{h_3'}{h_3^2}\Big)'-3h_1\frac{h_3'}{h_3^2}\Big]|\nabla u|^2+h_2\frac{h_3'}{h_3^2}\Big\}\Phi+\Delta^2 u\\
        \geqslant&~2(n-1)\kappa h_1|\nabla u|^2+\frac 2 n h_1\Big[\frac{\Phi}{h_3}+\Big(\frac{h_1'}{h_1}-2h_1\Big)|\nabla u|^2+h_2\Big]^2+\frac{h_1'}{h_3}|\nabla u|^2\Phi\\
        &~+\Big(h_1''-\frac{2(h_1')^2}{h_1}+h_1h_1'-2h_1^3\Big)|\nabla u|^4-\frac{h_2'}{h_3}\Phi+(h_1'h_2+3h_1h_2'-h_2'')|\nabla u|^2\\
        &~-h_2h_2'+\lambda_1\kappa\Delta u-f(u)\kappa^2+\Big\{\frac{h_3'}{h_3^3}\Phi+\Big[\Big(\frac{h_3'}{h_3^2}\Big)'-3h_1\frac{h_3'}{h_3^2}\Big]|\nabla u|^2+h_2\frac{h_3'}{h_3^2}\Big\}\Phi.
    \end{align*}
    Adding \(K\frac{h_1}{h_3^2}\Phi^2\) and rewriting as a quadratic form, we complete the proof.
\end{proof}

Thanks to Lemma \ref{lem:est}, we establish estimates involving \(\Delta u\) for equations \eqref{eq:Liouville} and \eqref{eq:subcritical}.

\begin{proposition}\label{prop:est-Liouville}
    Let \(u\) be a solution to \eqref{eq:Liouville}. If \(-2<\lambda_1<6\), \(\lambda_2>0\), then
    \[\Delta u+\frac{3(6-\lambda_1)(2+\lambda_1)}{8\lambda_2}|\nabla u|^2-\frac{4\lambda_2\kappa}{3(2+\lambda_1)}\leqslant0.\]
\end{proposition}

\begin{proof}
    For \(t\in[0,1]\), define \(u_t=tu+(1-t)\). Thus
    \[\Delta^2 u_t - \lambda_1\kappa\Delta u_t+ \lambda_2\kappa^2 = \lambda_2\kappa^2[t\mathrm e^{4u}+(1-t)]>0,\]
    which implies that \(u_t\) satisfies the differential inequality \eqref{ineq}. Take \(h_1(u)=a_1\), \(h_2(u)=a_2\kappa\), \(h_3(u)=\mathrm e^{a_3u}\) in Proposition \ref{lem:est}, where \(a_1,a_2>0\), \(a_3\in\mathbb R\), and \(\mathscr L_t\) is defined by \(\mathscr L\) w.r.t. \(u_t\). Thus for \(\Phi_t=\mathrm e^{a_3u_t} (\Delta u_t+a_1|\nabla u_t|^2-a_2\kappa)\), it follows from Proposition \ref{lem:est} that
    \[\mathrm e^{-a_3u_t}\mathscr L_t\Phi_t\geqslant a_1(\mathrm e^{-a_3u_t}\Phi_t,a_1|\nabla u_t|^2,a_2\kappa)A(\mathrm e^{-a_3u_t}\Phi_t,a_1|\nabla u_t|^2,a_2\kappa)^T,\]
    where \(A=(A_{ij})\) satisfies
    \[A_{11}=K+\frac 1 2+\frac{a_3}{a_1},\quad A_{12}=-\frac{(2a_1+a_3)(a_1+a_3)}{2a_1^2},\quad A_{13}=\frac{a_2(a_1+a_3)+\lambda_1}{2a_1a_2},\]
    \[A_{22}=0,\quad A_{23}=\frac{6-\lambda_1}{2a_1a_2}-1,\quad A_{33}=\frac 1 2+\frac{\lambda_1}{a_1a_2}-\frac{\lambda_2}{a_1a_2^2}.\]
    
    Take \(a_1=\frac{6-\lambda_1}{2a_2}=-a_3\), then \(A_{12}=A_{23}=0\), \(A_{33}>0\). Thus by choosing \(K\) sufficiently large, we ensure \(\mathscr L_t\Phi_t\geqslant0\). If \(\Phi_{t_0}\equiv 0\) for some \(t_0\in [0,1]\), then \(\mathrm{div}(\mathrm e^{a_1u_{t_0}} \nabla u_{t_0})=a_2\kappa\mathrm e^{a_1u_{t_0}}>0\). By integrating on \(M\), it contradicts the divergence theorem. Therefore, \(\Phi_t\not\equiv0\), \(\forall t\in[0,1]\). Combining with \(\Phi_0<0\), we conclude that \(\Phi_1<0\) from Lemma \ref{lem:continuity}, that is, 
    \[\Delta u+a_1|\nabla u|^2-a_2\kappa<0,\quad\forall(a_1,a_2)\in\Big\{a_1a_2=3-\frac{\lambda_1}{2},~a_2>\frac{4\lambda_2}{3(2+\lambda_1)}\Big\}.\]
    Let \(a_2\to\frac{4\lambda_2}{3(2+\lambda_1)}\), we complete the proof.
\end{proof}
\section{Fourth order Liouville equation: Proof of Theorem \ref{thm:Liouville}}\label{sec:Liouville}

In this section, \(n=4\), \(u\) is a solution to \eqref{eq:Liouville}, and the manifold \(M\), parameters \(\lambda_1,\lambda_2\) satisfy conditions in Theorem \ref{thm:Liouville}. We establish the differential identity using the invariant tensors technique. Using this identity, we prove Theorem \ref{thm:Liouville}. 

Define the trace-free tensor
\[E_{ij}:=u_{ij}+b u_i u_j-\frac{1}{4}(\Delta u+b|\nabla u|^2)g_{ij},\]
where \(b\) is a constant to be determined later. We trace the covariant derivative of \(E_{ij}\) to get
\[ {E_{ij,}}^i=\frac{3}{4} (\Delta u)_{,j}+R_{ij} u^i +b\Delta u u_j+\frac{b}{2} u_{ij} u^i.\]
Denoting that \(E_j= E_{ij} u^i\) and plugging the expression \(u_{ij}=E_{ij}-b u_i u_j+\frac{1}{4}(\Delta u+b|\nabla u|^2)g_{ij}\) into the above formula, we have
\begin{equation}\label{eq:divE}
{E_{ij,}}^i=\frac{b}{2}E_j+\frac{3}{4}F_j +R_{ij} u^i-3\kappa u_j,
\end{equation}
where \(F_j=(\Delta u)_{,j}+\frac{3b}{2}\Delta u u_j-\frac{b^2}{2}|\nabla u|^2 u_j+4\kappa u_j\). Thus,
\[{F_{i,}}^i=\Delta ^2 u+\frac{3b}{2} {(\Delta u)_,}^i u_i +\frac{3b}{2}(\Delta u)^2 -b^2 u^{ij} u_i u_j-\frac{b^2}{2} \Delta u |\nabla u|^2+4\kappa\Delta u.\]
Likewise, denote \(E= E_i u^i\), \(F= F_i u^i\) and replace \(u^{ij}\), \({(\Delta u)_,}^i\) by \(E^{ij}\), \(F^i\) respectively. Thus,
\begin{align}
{F_{i,}}^i=-b^2 E+\frac 3 2 b F+G, \label{eq:divF}
\end{align}
where \(G=\Delta ^2 u+\frac{3b}{2}(b|\nabla u|^2-\Delta u)^2-6b\kappa|\nabla u|^2+4\kappa\Delta u\).

To compute \({G_,}^i\), differentiating \eqref{eq:Liouville} gives that
\begin{align}\label{eq:nabladelta2u}
    {(\Delta ^2 u)_,}^i=~&\lambda_1\kappa {(\Delta u)_,}^i+4\Delta^2u u^i-4\lambda_1\kappa\Delta u u^i+4\lambda_2\kappa^2 u^i. 
\end{align}
Under further calculation, it yields that
\[{G_,}^i={(\Delta^2 u)_,}^i+3b(b|\nabla u|^2-\Delta u)[2bu^{ij}u_j-{(\Delta u)_,}^i]-12b\kappa u^{ij} u_j+4\kappa{(\Delta u)_,}^i.\]
Substituting \eqref{eq:nabladelta2u} to \({G_,}^i\) and replacing \(u^{ij}\), \({(\Delta u)_,}^i\), \(\Delta^2 u\) by \(E^{ij}\), \(F^i\), \(G\) respectively, we obtain
\begin{align}
    \begin{split}\label{eq:nablaG}
    {G_,}^i=~&3b(b|\nabla u|^2-\Delta u)(2bE^i-F^i)-12b\kappa E^i+(4+\lambda_1)\kappa F^i+4Gu^i\\
    &-6b(1+b)(b|\nabla u|^2-\Delta u)^2 u^i+b(24+23b+\frac b 2\lambda_1)\kappa|\nabla u|^2 u^i\\
    &-[16+21b+(4+\frac 3 2 b)\lambda_1]\kappa\Delta u u^i+4(\lambda_2-\lambda_1-4)\kappa^2 u^i.
    \end{split}
\end{align}
Take \(b=-1\) to eliminate the term \((b|\nabla u|^2-\Delta u)^2 u^i\). Thus, the following proposition states relationships between invariant tensors.

\begin{proposition}\label{prop:invariant-Liouville}
    Invariant tensors are chosen as
    \[E_{ij}=u_{ij}-u_i u_j-\frac{1}{4}(\Delta u-|\nabla u|^2)g_{ij},\]
    \[F_i=(\Delta u)_i-\frac{3}{2} \Delta u u_i-\frac{1}{2} |\nabla u|^2 u_i+4\kappa u_i,\]
    \[G=\Delta ^2 u-\frac{3}{2}(|\nabla u|^2+\Delta u)^2+6\kappa|\nabla u|^2+4\kappa\Delta u.\]
    Relationships between them via differentiation are as follows:
    \[{E_{i,}}^i=E_{ij}E^{ij}+\mathscr R+\frac{1}{2}E+\frac{3}{4}F,\quad{F_{i,}}^i=-E-\frac{3}{2}F+G,\]
    \begin{align*}
        {G_,}^i=~&-3(|\nabla u|^2+\Delta u)(2E^i+F^i)+12\kappa E^i+(4+\lambda_1)\kappa F^i+4Gu^i\\
        &-(1-\frac{\lambda_1}{2})(|\nabla u|^2-5\Delta u)\kappa u^i+4(\lambda_2-\lambda_1-4)\kappa^2 u^i,
    \end{align*}
    where \(\mathscr R:=R_{ij}u^iu^j-(n-1)\kappa|\nabla u|^2\geqslant0.\)
\end{proposition}

The key differential identity is established by Proposition \ref{prop:invariant-Liouville}.

\begin{proposition}
    Let \(c\) be a constant to be determined, then
    \begin{align}
        \begin{split}\label{id:Liouville}
            &\mathrm e^u\Big\{\mathrm e^{-u}\Big[-\frac 2 3[|\nabla u|^2+7\Delta u-2(8+\lambda_1)\kappa]E_i+(3|\nabla u|^2+\Delta u)F_i\\
            &~~~~~~~~~~~~~~~{-Gu_i+c(2-\lambda_1)\kappa|\nabla u|^2u_i\Big]\Big\}_{,}}^i\\
            =~&-\frac 2 3[|\nabla u|^2+7\Delta u-2(8+\lambda_1)\kappa](E_{ij}E^{ij}+\mathscr R)+(F_i+\frac 2 3E_i)(F^i+\frac 2 3E^i)\\
            &-\frac{16}{9}E_iE^i+\frac 1 2(2-\lambda_1)\kappa[(1+c)|\nabla u|^4-(5-3c)|\nabla u|^2\Delta u]\\
            &+2(\frac 1 3+c)(2-\lambda_1)\kappa E+4(4+\lambda_1-\lambda_2)\kappa^2|\nabla u|^2.
        \end{split}
    \end{align}
\end{proposition}

\begin{proof}
    By Proposition \ref{prop:invariant-Liouville}, we derive the following differential identities:
    \[\mathrm e^u{(\mathrm e^{-u}|\nabla u|^2 E_i)_{,}}^{i}=|\nabla u|^2(E_{ij}E^{ij}+\mathscr R)+2E_iE^i+|\nabla u|^2 E+\frac 1 2 \Delta u E+\frac{3}{4}|\nabla u|^2 F,\]
    \[\mathrm e^u{(\mathrm e^{-u} \Delta u E_i)_{,}}^i=\Delta u(E_{ij}E^{ij}+\mathscr R)+E_iF^i +\frac 1 2|\nabla u|^2 E +\Delta u E-4\kappa E+\frac{3}{4}\Delta u F,\]
    \[\mathrm e^u{(\mathrm e^{-u} E_i)_{,}}^i=E_{ij}E^{ij} +\mathscr R-\frac 1 2 E +\frac 3 4 F,\]
    \[\mathrm e^u{(\mathrm e^{-u}|\nabla u|^2 F_i)_{,}}^i=2E_i F^i-|\nabla u|^2 E -|\nabla u|^2 F+\frac 1 2 \Delta u F+|\nabla u|^2 G,\]
    \[\mathrm e^u{(\mathrm e^{-u} \Delta u F_i)_{,}}^i=F_iF^i -\Delta u E+\frac 1 2|\nabla u|^2 F-\Delta u F -4\kappa F+\Delta u G,\]
    \begin{align*}
        \mathrm e^u{(\mathrm e^{-u} G u_i)_{,}}^i=&-3(|\nabla u|^2+\Delta u)(2E+F)+12\kappa E+(4+\lambda_1)\kappa F+(3|\nabla u|^2+\Delta u)G\\
        &-(1-\frac{\lambda_1}{2})\kappa(|\nabla u|^2-5\Delta u)|\nabla u|^2+4(\lambda_2-\lambda_1-4)\kappa^2|\nabla u|^2,
    \end{align*}
    \[\mathrm e^u{(\mathrm e^{-u}|\nabla u|^2u_i)_{,}}^i=2E+\frac 1 2|\nabla u|^4+\frac 3 2|\nabla u|^2\Delta u.\]
    We finish the proof by linearly combining seven identities.
\end{proof}

\begin{proof}[\normalfont\bfseries Proof of Theorem \ref{thm:Liouville}]
    Assume that \(u\) isn't constant. Define \(L_{ij}=u_iu_j-\frac 1 4|\nabla u|^2g_{ij}\), then
    \[g_{ij}L^{ij}=0,\quad E_{ij}L^{ij}=E,\quad L_{ij}L^{ij}=\frac 3 4|\nabla u|^4.\]
    
    We claim that if \(-2<\lambda_1\leqslant2\), \(0<\lambda_2\leqslant L_1(\lambda_1)\), then the condition
    \begin{equation}
        \label{cond:Liouville-1}
        14\lambda_2 \leqslant 3(8+\lambda_1)(2+\lambda_1) 
    \end{equation}
    holds. We'll verify \eqref{cond:Liouville-1} at the end of each case. The condition \eqref{cond:Liouville-1} yields
    \begin{equation}\label{cond:Liouville-2}
        8\lambda_2<7(6-\lambda_1)(2+\lambda_1)\leqslant 21(6-\lambda_1)(2+\lambda_1)
    \end{equation}
    by verifying \(7(6-\lambda_1)(2+\lambda_1)-8\lambda_2\geqslant\frac 1 7(2+\lambda_1)(198-61\lambda_1)>0\).
    
    By Proposition \ref{prop:est-Liouville}, \(h:=\Delta u+\frac{3(6-\lambda_1)(2+\lambda_1)}{8\lambda_2}|\nabla u|^2-\frac{4\lambda_2\kappa}{3(2+\lambda_1)}\leqslant0\). Similar as the proof in \cite[Lemma 2.2]{MOW25}, we have \(|\nabla u|^2 E_{ij} E^{ij}\geqslant \frac 4 3 E_i E^i\). Thus, by \eqref{cond:Liouville-2},
    \begin{align}
        \begin{split}\label{ineq:Liouville-1}
            &-\frac 2 3[|\nabla u|^2+7\Delta u-2(8+\lambda_1)\kappa]E_{ij}E^{ij}-\frac{16}{9}E_iE^i\\
            \geqslant~&-\frac 2 3\Big[\Big(1-\frac{21(6-\lambda_1)(2+\lambda_1)}{8\lambda_2}\Big)|\nabla u|^2\\
            &~~~~~~~~~~+\Big(\frac{28\lambda_2}{3(2+\lambda_1)}-2(8+\lambda_1)\Big)\kappa\Big]E_{ij}E^{ij}-\frac{16}{9}E_iE^i\\
            \geqslant~&\frac 4 3\Big(8+\lambda_1-\frac{14\lambda_2}{3(2+\lambda_1)}\Big)\kappa E_{ij}E^{ij}+\frac 1 3\Big(\frac{7(6-\lambda_1)(2+\lambda_1)}{\lambda_2}-8\Big)E_iE^i\\
            \geqslant~&\frac 4 3\Big(8+\lambda_1-\frac{14\lambda_2}{3(2+\lambda_1)}\Big)\kappa E_{ij}E^{ij}.
        \end{split}
    \end{align}
    If equalities hold in \eqref{ineq:Liouville-1}, then \(E_i=0\) by \eqref{cond:Liouville-2}. By \eqref{ineq:Liouville-1} and the Divergence Theorem, \eqref{id:Liouville} yields
    \begin{align}
        \begin{split}\label{ineq:Liouville-2}
            0\geqslant~&\int_M\mathrm e^{-u}\Big[\frac 4 3\Big(8+\lambda_1-\frac{14\lambda_2}{3(2+\lambda_1)}\Big)E_{ij}E^{ij}+\frac 1 2(2-\lambda_1)[(1+c)|\nabla u|^4\\
            &~~~~~~~~~~~~~~-(5-3c)|\nabla u|^2\Delta u]+2(\frac 1 3+c)(2-\lambda_1)E+4(4+\lambda_1-\lambda_2)\kappa|\nabla u|^2\Big].
        \end{split}
    \end{align}
    If equality holds in \eqref{ineq:Liouville-2}, then \(E_i=F_i=0\) and \(\mathscr R=0\).
    
    \textbf{Case 1: \(\frac{2}{129}(\sqrt{18673}-64)\leqslant \lambda_1\leqslant 2\) and \(0<\lambda_2\leqslant4+\lambda_1\).} Take \(c=\frac 5 3\), then \eqref{ineq:Liouville-2} yields
    \begin{align}
        \begin{split}\label{ineq:Liouville-3}
            0\geqslant~&\int_M\mathrm e^{-u}\Big[\frac 4 3\Big(8+\lambda_1-\frac{14\lambda_2}{3(2+\lambda_1)}\Big)E_{ij}E^{ij}+\frac 4 3(2-\lambda_1)|\nabla u|^4\\
            &~~~~~~~~~~~~~~+4(2-\lambda_1)E+4(4+\lambda_1-\lambda_2)\kappa|\nabla u|^2\Big]\\
            =~&\int_M\mathrm e^{-u}\Big[(2-\lambda_1)\Big(\frac 3 2E_{ij}+\frac 4 3L_{ij}\Big)\Big(\frac 3 2E^{ij}+\frac 4 3L^{ij}\Big)\\
            &~~~~~~~~~~~~~~+\Big(\frac{74+43\lambda_1}{12}-\frac{56\lambda_2}{9(2+\lambda_1)}\Big)E_{ij}E^{ij}+4(4+\lambda_1-\lambda_2)\kappa|\nabla u|^2\Big].
        \end{split}
    \end{align}
    By \(\lambda_1\geqslant\frac{2}{129}(\sqrt{18673}-64)\) and \(\lambda_2\leqslant4+\lambda_1\),
    \begin{align*}
        \frac{74+43\lambda_1}{12}-\frac{56\lambda_2}{9(2+\lambda_1)}\geqslant\frac{74+43\lambda_1}{12}-\frac{56(4+\lambda_1)}{9(2+\lambda_1)}=\frac{129\lambda_1^2+256\lambda_1-452}{36(2+\lambda_1)}\geqslant0.
    \end{align*}
    Thus the equality must hold in \eqref{ineq:Liouville-3}, implying that \(E_i=F_i=0\), \(\mathscr R=0\) and \(\lambda_2=4+\lambda_1\).
    
    If \(\lambda_1<2\), then \(\frac 3 2E_{ij}+\frac 4 3L_{ij}=0\). We obtain
    \[0=(\frac 3 2E_{ij}+\frac 4 3L_{ij})u^j=\frac 3 2E_i+|\nabla u|^2u_i=|\nabla u|^2u_i,\]
    which is contradict to \(u\) is non-constant.
    
    If \(\lambda_1=2\), \(\lambda_2=6\), then \(E_{ij}=0\). Thus, the non-constant function \(\mathrm e^{-u}\) satisfies \eqref{ineq:rigidity}. By Lemma \ref{lem:rigidity}, \((M^4,g)\) is isometric to \((\mathbb{S}^4(\frac{1}{\sqrt{\kappa}}),g_c)\), and for some \(a\in \mathbb{R}^5 \setminus \{0\}\), \(c\in \mathbb R\),
    \[u(x)=-\log(\langle a,x\rangle_{\mathbb{R}^5} +c),~x\in {\mathbb{S}}^4(\frac{1}{\sqrt{\kappa}}).\]
    Integrating the equation \eqref{eq:Liouville} yields \(\int_{{\mathbb{S}}^4(\frac{1}{\sqrt{\kappa}})} \mathrm e^{4u} =1\). Thus,
    \begin{align*}
        1&=\int_{{\mathbb{S}}^4(\frac{1}{\sqrt{\kappa}})} (\langle a,x\rangle_{\mathbb{R}^5} +c)^{-4} =\frac{\kappa^2}{\omega_5} \int_{{\mathbb{B}}^5(\frac{1}{\sqrt{\kappa}})} \mathrm{div}[(\langle a,x\rangle_{\mathbb{R}^5} +c)^{-4} \sqrt{\kappa}x]\mathrm{d}x \\
        &=\frac{\kappa^{\frac{5}{2}}}{\omega_5} \int_{{\mathbb{B}}^5(\frac{1}{\sqrt{\kappa}})} \frac{|a|x_5 +5c}{(|a|x_5 +c)^5}\mathrm{d}x = \frac{\kappa^2 \omega_4}{4 \omega_5} \int_{-1}^1  \frac{|a|x_5 +5\sqrt{\kappa}c}{(|a|x_5+\sqrt{\kappa}c)^5}(1-x_5^2)^2 \mathrm{d}x_5.
    \end{align*}
    Notice that \(\sqrt{\kappa}|c|\geqslant |a|\) is required for the convergence of the integral, then
    \[1=\frac{4\kappa^2 \omega_4}{3\omega_5 (|a|^2-\kappa c^2)^2} =\frac{\kappa^2}{(|a|^2-\kappa c^2)^2}.\]
    It follows that \(c^2-\kappa^{-1}|a|^2=1\). Hence, the classification result \eqref{sol:Liouville} is obtained.
    
    Finally, if \(1<\frac{2}{129}(\sqrt{18673}-64)\leqslant\lambda_1\leqslant2\), \(\lambda_2\leqslant4+\lambda_1\), the condition \eqref{cond:Liouville-1} is verified by
    \[3(8+\lambda_1)(2+\lambda_1)-14\lambda_2\geqslant3\lambda_1^2+16\lambda_1-8>11.\]
    
    \textbf{Case 2: \(-2<\lambda_1<\frac{2}{129}(\sqrt{18673}-64)\) and \(0<\lambda_2\leqslant L_1(\lambda_1)\).} Take
    \[c=\frac 5 3-\frac{2(2+\lambda_1)(4+\lambda_1-\lambda_2)}{(2-\lambda_1)\lambda_2}<\frac 5 3.\]
    By replacing \(\Delta u\) with \(h\), \eqref{ineq:Liouville-2} becomes
    \begin{align}
        \begin{split}\label{ineq:Liouville-4}
            0\geqslant~&\int_M\mathrm e^{-u}\Big\{\frac 4 3\Big(8+\lambda_1-\frac{14\lambda_2}{3(2+\lambda_1)}\Big)E_{ij}E^{ij}+(2-\lambda_1)\Big[-\frac{5-3c}{2}|\nabla u|^2h\\
            &~~~~~~~~~~~~~~+\Big(\frac 4 3+\frac{(2+\lambda_1)(4+\lambda_1-\lambda_2)}{8(2-\lambda_1)\lambda_2^2}[9(6-\lambda_1)(2+\lambda_1)-8\lambda_2]\Big)|\nabla u|^4\Big]\\
            &~~~~~~~~~~~~~~+4\Big(2-\lambda_1-\frac{(2+\lambda_1)(4+\lambda_1-\lambda_2)}{\lambda_2}\Big)E\Big\}\\
            \geqslant~&\frac 4 3\int_M\mathrm e^{-u}\Big\{\Big(8+\lambda_1-\frac{14\lambda_2}{3(2+\lambda_1)}\Big)E_{ij}E^{ij}+\Big(\frac 4 3(2-\lambda_1)\\
            &~~~~~~~~~~~~~~+\frac{(2+\lambda_1)(4+\lambda_1-\lambda_2)}{8\lambda_2^2}[9(6-\lambda_1)(2+\lambda_1)-8\lambda_2]\Big)L_{ij}L^{ij}\Big]\\
            &~~~~~~~~~~~~~~+3\Big(2-\lambda_1-\frac{(2+\lambda_1)(4+\lambda_1-\lambda_2)}{\lambda_2}\Big)E\Big\}.
        \end{split}
    \end{align}
    By condition \eqref{cond:Liouville-1}, \(9(6-\lambda_1)(2+\lambda_1)-8\lambda_2\geqslant\frac 3 7(2+\lambda_1)(94-25\lambda_1)>0\), implying that the coefficient of \(L_{ij}L^{ij}\) is positive. The discriminant of the quadratic trinomial is
    \begin{align*}
        &9\Big(2-\lambda_1-\frac{(2+\lambda_1)(4+\lambda_1-\lambda_2)}{\lambda_2}\Big)^2-4\Big(8+\lambda_1-\frac{14\lambda_2}{3(2+\lambda_1)}\Big)\Big(\frac 4 3(2-\lambda_1)\\
        &+\frac{(2+\lambda_1)(4+\lambda_1-\lambda_2)}{8\lambda_2^2}[9(6-\lambda_1)(2+\lambda_1)-8\lambda_2]\Big)=\frac{3(2+\lambda_1)^2}{2\lambda_2^2}Q_1\Big(\frac{\lambda_2}{3(2+\lambda_1)}\Big),
    \end{align*}
    where \(Q_1(x)=112(14-\lambda_1) x^3-2(996+332\lambda_1-67\lambda_1^2)x^2\)
    \[+(1552+432\lambda_1-70\lambda_1^2-9\lambda_1^3)x-3(4+\lambda_1)(40-4\lambda_1-\lambda_1^2).\]
    The function \(Q_1\) is strictly increasing, because the discriminant of
    \[Q_1'(x)=336(14-\lambda_1) x^2-4(996+332\lambda_1-67\lambda_1^2)x+1552+432\lambda_1-70\lambda_1^2-9\lambda_1^3\]
    is negative when \(-2<\lambda_1\leqslant\frac{2}{129}(\sqrt{18673}-64)<1.13\), verified by
    \begin{align*}
        &16(-833136+283680\lambda_1+95368\lambda_1^2-39784\lambda_1^3+3733\lambda_1^4) \\
        <&16(-833136+283680\cdot 1.13+95368\cdot 1.13^2+39784\cdot 2^3+3733\cdot 2^4)=-204835.2128<0.
    \end{align*}
    Thus, by \(Q_1\Big(\frac{L(\lambda_1)}{3(2+\lambda_1)}\Big)=0\), the range \(0<\lambda_2\leqslant L_1(\lambda_1)\) is equivalent to \(Q_1\Big(\frac{\lambda_2}{3(2+\lambda_1)}\Big)\leqslant0\). When the discriminant is non-positive, the equality holds in \eqref{ineq:Liouville-4}, implying that \(E_{ij}=c_0L_{ij}\) for some constant \(c_0\neq0\), \(E_i=F_i=0\) and \(\mathscr R=0\). Thus,
    \[\frac 3 4c_0|\nabla u|^2 u_i=c_0L_{ij}u^j=c_0E_i=0,\]
    which is contradict to \(u\) is non-constant.
    
    Finally, we verify the condition \eqref{cond:Liouville-1}. By \(Q_1\Big(\frac{8+\lambda_1}{14}\Big)=\frac{6}{49}(20-\lambda_1)^2>0\) and \(Q_1\) is strictly increasing, we obtain \(\frac{8+\lambda_1}{14}>\frac{L(\lambda_1)}{3(2+\lambda_1)}\geqslant\frac{\lambda_2}{3(2+\lambda_1)}\).
\end{proof}
\section{Geometric inequality and the rigidity results: Proof of Theorem \ref{thm:Onofri}}\label{sec:Onofri}
In this section, \(\Lambda_1\) denotes the principal eigenvalue of \(-\Delta\). We first state the Poincaré inequality established by the orthogonal decomposition of the eigenfunctions of \(-\Delta\).
\begin{lemma}[\normalfont\bfseries Poincaré inequality] \label{lem:Poincare}
    \begin{equation}\label{ineq:Poincaré-H1}
        \int_M|\nabla f|^2\geqslant \Lambda_1\Big[\int_M f^2 -\Big(\int_M f\Big)^2\Big],~\forall f\in H^1 (M)
    \end{equation}
    \begin{equation}\label{ineq:Poincaré-H2}
        \int_M|\Delta f|^2\geqslant \Lambda_1 \int_M |\nabla f|^2,~\forall f\in H^2 (M)
    \end{equation}
    The equality holds if and only if \(-\Delta f=\Lambda_1\Big(f-\int_M f\Big)\).
\end{lemma}

The following rigidity lemma is useful to deduce the lower bound of \(\Lambda_1\).

\begin{lemma}[\normalfont\bfseries Rigidity lemma] \label{lem:rigidity}
    Suppose that a non-constant function \(f\in H^2(M)\) satisfies
    \begin{equation}\label{ineq:rigidity}
        \operatorname{Ric}(\nabla f,\nabla f)=(n-1)\kappa|\nabla f|^2,~\nabla^2f=\frac{\Delta f}{n}g,
    \end{equation}
    then \((M^n,g)\) is isometric to \(\Big(\mathbb{S}^n(\frac{1}{\sqrt{\kappa}}),g_c\Big)\). Moreover, \( f-\int_{\mathbb{S}^n(\frac{1}{\sqrt{\kappa}})} f\) is the principal eigenfunction of \(-\Delta_{\mathbb{S}^n(\frac{1}{\sqrt{\kappa}})}\), namely
    \[f(x)=\langle a,x\rangle_{\mathbb{R}^{n+1}} +c,\quad x\in {\mathbb{S}}^n(\frac{1}{\sqrt{\kappa}}),\quad a\in \mathbb{R}^{n+1} \setminus \{0\},\quad c\in \mathbb{R}.\]
\end{lemma}

\begin{proof}
    Since \eqref{ineq:rigidity} holds for any \(f+C\), where \(C\) is a constant, we assume that \( \int_M f=0\) without loss of generality. Using the Bochner formula, we have
    \[0=\int_M |\nabla^2 f|^2-\int_M |\Delta f|^2 + \int_M \operatorname{Ric}(\nabla f,\nabla f)=\frac{n-1}{n}\Big(n\kappa \int_M|\nabla f|^2- \int_M |\Delta f|^2 \Big),\]
    then \(\Lambda_1\leqslant n\kappa\) is obtained coupling \eqref{ineq:Poincaré-H2} with the fact that \(f\) is non-constant. Furthermore, \(\Lambda_1 = n\kappa\) is ensured by Lemma \ref{lem:lower bound}. For the statement of eigenfunction, \eqref{ineq:Poincaré-H1} yields
    \begin{align*}
        \int_M |\Delta f+n\kappa f|^2&=\int_M |\Delta f|^2-2n\kappa \int_M |\nabla f|^2 + n^2 \kappa^2 \int_M f^2\\
        &=-n\kappa \Big( \int_M|\nabla f|^2- n\kappa\int_M f^2 \Big)\leqslant 0.
    \end{align*}
    We immediately get \(-\Delta f=n\kappa f\) and complete the proof by using Theorem 2 in \cite{Oba62}.
\end{proof}

\begin{lemma}[\normalfont\bfseries The lower bound of the principal eigenvalue]\label{lem:lower bound}
    \(\Lambda_1\geqslant n\kappa\). In particular,
    \begin{equation}\label{ineq:Poincaré-est}
        \int_M|\Delta f|^2\geqslant n\kappa\int_M|\nabla f|^2,~\forall f\in H^2(M) 
    \end{equation}
    unless \(f\) is constant. The equality holds iff \((M^n,g)\) is isometric to \(\Big(\mathbb{S}^n(\frac{1}{\sqrt{\kappa}}),g_c\Big)\), and
    \[f(x)=\langle a,x\rangle_{\mathbb{R}^{n+1}} +c,\quad x\in {\mathbb{S}}^n(\frac{1}{\sqrt{\kappa}}),\quad a\in \mathbb{R}^{n+1} \setminus \{0\},\quad c\in \mathbb{R}.\]
\end{lemma}

\begin{proof}
    The Bochner formula gives \(\int_M|\nabla^2 f|^2-\int_M|\Delta f|^2+\int_M\operatorname{Ric}(\nabla f,\nabla f)=0\). Notice that
    \[|\nabla^2 f|^2=\Big|\nabla^2 f-\frac{\Delta f}{n}g\Big|^2+\frac 1 n|\Delta f|^2\geqslant\frac 1 n|\Delta f|^2.\]
    Combing with \(\operatorname{Ric}\geqslant (n-1)\kappa g\), we deduce \eqref{ineq:Poincaré-est}. Replacing \(f\) by the non-constant function which ensures that the equality in \eqref{ineq:Poincaré-H2} holds, we obtain that \(\Lambda_1\geqslant n\kappa\). To conclude, the rigidity statements follow from Lemma \ref{lem:rigidity}, thanks to \eqref{ineq:rigidity}.
\end{proof}
\begin{remark}
    We prove the lower bound of the principal eigenvalue in Lemma \ref{lem:lower bound} without the equality holding condition in Lemma \ref{lem:rigidity}. Therefore, the above proof makes sense.
\end{remark}

\begin{proof}[\normalfont\bfseries Proof of Theorem \ref{thm:Onofri}]

    We prove Theorem \ref{thm:Onofri} by giving \eqref{ineq:Onofri-lambda} and its rigidity.

    For \(f\in C^4(\mathbb{S}^4(\frac{1}{\sqrt{\kappa}})),~\lambda_1 \in \mathbb R\), consider the functional
    \[J[f;\lambda_1]=\int_{\mathbb{S}^4(\frac{1}{\sqrt{\kappa}})}|\Delta f|^2+\lambda_1\kappa\int_{\mathbb{S}^4(\frac{1}{\sqrt{\kappa}})}|\nabla f|^2-8(4+\lambda_1)\kappa^2\Big(\log\int_{\mathbb{S}^4(\frac{1}{\sqrt{\kappa}})} \mathrm e^f-\int_{\mathbb{S}^4(\frac{1}{\sqrt{\kappa}})} f\Big).\]
    From Theorem 1 in \cite{Bec93} we have that \(J[f;2]\geqslant 0\), and \(J[f;-4]\geqslant 0\) corresponds to \eqref{ineq:Poincaré-H2} in Lemma \ref{lem:Poincare}. Then \(J[f;\lambda_1]\geqslant 0,~-4\leqslant \lambda_1 \leqslant 2\) follows from the linearity of \(J\) w.r.t. \(\lambda_1\). Combining the Jensen's inequality \(\int_{\mathbb{S}^4(\frac{1}{\sqrt{\kappa}})} f \leqslant \log\int_{\mathbb{S}^4(\frac{1}{\sqrt{\kappa}})} \mathrm e^f\) with \(\lambda_2\leqslant 4+ \lambda_1\), we show \eqref{ineq:Onofri-lambda}.

    For the rigidity statement, we calculate the variation of \(J[f]\) for any \(g\in C^{\infty}(M)\):
    \begin{align*}
        \frac{\mathrm d}{\mathrm d\varepsilon}\Big|_{\varepsilon=0}J[f+\varepsilon g;\lambda_1]=&~2\int_{\mathbb{S}^4(\frac{1}{\sqrt{\kappa}})}\Delta f\Delta g+2\lambda_1\kappa\int_{\mathbb{S}^4(\frac{1}{\sqrt{\kappa}})}\nabla f\cdot\nabla g\\
        &-8(4+\lambda_1)\kappa^2\Big[\Big(\int_{\mathbb{S}^4(\frac{1}{\sqrt{\kappa}})}\mathrm e^f\Big)^{-1}\int_{\mathbb{S}^4(\frac{1}{\sqrt{\kappa}})}\mathrm e^fg-\int_{\mathbb{S}^4(\frac{1}{\sqrt{\kappa}})} g\Big].
    \end{align*}
    Consider \(\frac{\mathrm d}{\mathrm d\varepsilon}\Big|_{\varepsilon=0}J[f+\varepsilon g;\lambda_1]=0\) and we obtain that \(f\) satisfies the equation
    \[\Delta^2f-\lambda_1\kappa\Delta f+4(4+\lambda_1)\kappa^2\Big[1-\Big(\int_{\mathbb{S}^4(\frac{1}{\sqrt{\kappa}})}\mathrm e^f\Big)^{-1}\mathrm e^f\Big]=0.\]
    Since it holds for any \(f+C\), where \(C\) is a constant, we assume that \( \int_{\mathbb{S}^4(\frac{1}{\sqrt{\kappa}})} \mathrm e^f=1\) without loss of generality, implying that \(u=\frac{f}{4}\) satisfies Liouville equation \eqref{eq:Liouville} with \(\lambda_2=4+\lambda_1\). By Theorem \ref{thm:Liouville}, for \(\frac{3}{2}\leqslant\lambda_1<2\), \(f\) is constant, which gives the conditions when the equality in \(J[f;\lambda_1]\geqslant 0\) is attained. The rigidity of \(J[f;2] \geqslant 0 \) follows by the classification result in Theorem \ref{thm:Liouville} similarly.
    
    For \(\lambda_1\in(-4,\frac{3}{2})\), there exists \(\theta\in(0,1)\) such that 
    \[J[f;\lambda_1]=\theta J[f;\frac{3}{2}]+(1-\theta) J[f;-4].\]
    By rigidity of \(J[f;\frac{3}{2}] \geqslant 0 \), we conclude that for \(-4<\lambda_1<2\) and \(0<\lambda_2<4+\lambda_1\), the equality in \eqref{ineq:Onofri-lambda} is attained iff \(f\) is constant. By the Jensen's inequality, \(\int_{\mathbb{S}^4(\frac{1}{\sqrt{\kappa}})} f \leqslant \log\int_{\mathbb{S}^4(\frac{1}{\sqrt{\kappa}})} \mathrm e^f\) holds iff \(f\) is constant. Thus we obtain the rigidity result for \(-4<\lambda_1\leqslant 2\) and \(0<\lambda_2<4+\lambda_1\).
\end{proof}

\textbf{Acknowledgments.} Xi-Nan Ma was supported by the National Natural Science Foundation of China (Grant No. 12141105) and the National Key Research and Development Project (Grant No. SQ2020YFA070080). Tian Wu was supported by Anhui Postdoctoral Scientific Research Program Foundation (Grant No. 2025B1055).


\footnotesize{
    Welcome contact us:
    \begin{itemize}
        \item Xi-Nan Ma, School of Mathematical Sciences, University of Science and Technology of China, Hefei, Anhui, 230026, People's Republic of China. Email: \emph{xinan@ustc.edu.cn}
        \item Tian Wu, School of Mathematical Sciences, University of Science and Technology of China, Hefei, Anhui, 230026, People's Republic of China. Email: \emph{wt1997@ustc.edu.cn}
        \item Xiao Zhou, School of Mathematical Sciences, University of Science and Technology of China, Hefei, Anhui, 230026, People's Republic of China. Email: \emph{ZX238000@mail.ustc.edu.cn}
    \end{itemize}
}


\begin{thebibliography}{99}\small
    \bibitem{Bec93} W. Beckner, Sharp Sobolev inequalities on the sphere and the Moser–Trudinger inequality, Ann. of Math. 138 (1) (1993) 213–242.
    
    \bibitem{VV91} M.-F. Bidaut-Véron, L. Véron, Nonlinear elliptic equations on compact Riemannian manifolds and asymptotics of Emden equations, Invent. Math. 106 (3) (1991) 489–539.
    
    \bibitem{Bra85} T. P. Branson, Differential operators canonically associated to a conformal structure, Math. Scand. 57 (1985) 293–345.
    
    \bibitem{CGS89} L. A. Caffarelli, B. Gidas, J. Spruck, Asymptotic symmetry and local behavior of semilinear elliptic equations with critical Sobolev growth, Comm. Pure Appl. Math. 42 (3) (1989) 271–297.
    
    \bibitem{Cas24} J. Case, The Obata–Vétois argument and its applications, J. Reine Angew. Math. (Crelles Journal) 815 (2024) 23–40.
    
    \bibitem{CL91} W. Chen, C. Li, Classification of solutions of some nonlinear elliptic equations, Duke Math. J. 63 (3) (1991) 615–622.
    
    \bibitem{DEJ15} J. Dolbeault, M. J. Esteban, G. Jankowiak, The Moser-Trudinger-Onofri inequality, Chin. Ann. Math. 36 (6) (2015) 777–802.

    \bibitem{DEJ16} J. Dolbeault, M. J. Esteban, G. Jankowiak, Onofri inequalities and rigidity results, Discrete Contin. Dyn. Syst. 37 (2017), 3059–3078.
    
    \bibitem{GS81} B. Gidas, J. Spruck, Global and local behavior of positive solutions of nonlinear elliptic equations, Commun. Pure Appl. Math. 34 (4) (1981) 525–598.
    
    \bibitem{LW25} M. Li, J. Wei, A remark on the Case-Gursky-Vétois identity and its applications, Proc. Am. Math. Soc. 153 (8) (2025) 3417–3430.
    
    \bibitem{Lin98} C. Lin, A classification of solutions of a conformally invariant fourth order equation in $\mathbb{R}^n$, Comment. Math. Helv. 73 (2) (1998) 206–231.
    
    \bibitem{MOW25} X. Ma, Q. Ou, T. Wu, Jerison-Lee identities and semi-linear subelliptic equations on Heisenberg group, Acta Math. Sci. 45 (1) (2025) 264–279.
    
    \bibitem{MW24} X. Ma, T. Wu, The application of the invariant tensor technique in the classification of solutions to semilinear elliptic and sub-elliptic partial differential equations (in Chinese), Sci. Sin. Math. 54 (2024) 1627–1648.
    
    \bibitem{MWW25} X. Ma, T. Wu, W. Wu, Liouville theorem of the subcritical semilinear equation of fourth order on complete manifolds, preprint.
    
    \bibitem{MWZ25} X. Ma, T. Wu, X. Zhou, The subcritical biharmonic equation and Sobolev embedding nequality on closed 4-manifolds, preprint, (2025).
    
    
    \bibitem{Oba62} M. Obata, Certain conditions for a Riemannian manifold to be isometric with a sphere, J. Math. Soc. Japan 14 (1962) 333–340.
    
    \bibitem{Oba71} M. Obata, The conjectures on conformal transformations of Riemannian manifolds, J. Differential Geometry 6 (1971) 247–258.
    
    \bibitem{Ono82} E. Onofri, On the positivity of the effective action in a theory of random surfaces, Comm. Math. Phys. 86 (3) (1982) 321–326.
    
    \bibitem{Vét24} J. Vétois, Uniqueness of conformal metrics with constant Q-curvature on closed Einstein manifolds, Potential Anal. 61 (3) (2024) 485–500.
    
\end{thebibliography}
\end{document}